\documentclass[reqno,12pt,centertags]{amsart}
\usepackage{geometry}
\usepackage{amsmath,amsthm,amscd,amssymb,latexsym,upref,stmaryrd}
\usepackage{graphicx}
\usepackage{hyperref}
\usepackage{enumitem}
\usepackage{xcolor}
\newcommand*{\mailto}[1]{\href{mailto:#1}{\nolinkurl{#1}}}
\usepackage[numbers,sort&compress]{natbib}

\newcommand{\beq}{\begin{equation}}
	\newcommand{\eeq}{\end{equation}}
\newcommand{\ba}{\begin{align}}
	\newcommand{\ea}{\end{align}}

\renewcommand{\ln}{\text{\rm ln}}

 \allowdisplaybreaks
\numberwithin{equation}{section}
\newtheorem{theorem}{Theorem}[section]
\newtheorem{lemma}[theorem]{Lemma}

\theoremstyle{definition}

\begin{document}
	
	%\Volume{}
	%\Year{}
	%\DOIsuffix{mana.200310}
	%\pagespan{1}{}
	%\Receiveddate{}
	%\Reviseddate{}
	%\Accepteddate{}
	%\Dateposted{}
	
	\title[Inverse problems for Dirac operators]
	{Inverse problems for Dirac operators with constant delay less than half of the interval}
	
	\author[F.~Wang]{Feng Wang}
	\address{School of Mathematics and Statistics, Nanjing University of
		Science and Technology, Nanjing, 210094, Jiangsu, China}
	\email{\mailto{wangfengmath@njust.edu.cn}}
	
	\author[C.~F.~Yang]{CHUAN-FU Yang}
	\address{Department of Mathematics, School of Mathematics and Statistics, Nanjing University of
		Science and Technology, Nanjing, 210094, Jiangsu, People's
		Republic of China}
	\email{\mailto{chuanfuyang@njust.edu.cn}}

	\date{\today}
	
	\begin{abstract}
{In this work, we consider Dirac-type operators with a constant delay less than half of the interval and not less than two-fifths of the interval. For our considered Dirac-type operators, two inverse spectral problems are studied. Specifically, reconstruction of two complex $L_{2}$-potentials is studied from complete spectra of two boundary value problems with one common Dirichlet boundary condition and Neumann boundary condition, respectively. We give answers to the full range of questions usually raised in the inverse spectral theory. That is, we give uniqueness, necessary and sufficient conditions of the solvability, reconstruction algorithm and uniform stability for our considered inverse problems.}
	\end{abstract}
	
	\maketitle
	
\section{introduction}
	In the past decade, there appeared a significant interest in inverse problems for Sturm-Liouville-type operators with constant delay:
\begin{equation}\label{0.1}
  -y''(x)+q(x)y(x-a)=\lambda y(x),\quad x\in(0,\pi),
\end{equation}
under two-point boundary conditions (see [2-3, 8-13, 15, 20, 24-26, 29-30] and references therein),  which are often adequate for modelling various real-world processes frequently possessing a nonlocal nature. Here $q(x)$ is a complex-valued function in $L_{2}(a,\pi)$ vanishing on $(0,a)$. It is well known that the potential $q(x)$ is uniquely determined by specifying the spectra of two boundary value problems for equation (\ref{0.1}) with a common boundary condition at zero as soon as $a\in[\frac{2\pi}{5},\pi)$. The recent series of papers [10-12] establishes, however, that it is never possible for $a\in(0,\frac{2\pi}{5})$. For more details, see Introduction in \cite{B-M-S}.

To the best of our knowledge, the first attempt of defining Dirac-type operator with constant delay was made in \cite{Buter2}. They consider the following Dirac-type system with a delay constant $a\!\in\!(0,\pi)$:
\begin{align}\label{1.1}
  By'(x)+Q(x)y(x-a)=\lambda y(x),\quad x\in (0,\pi),
\end{align}
where
\begin{align}
  B=\begin{bmatrix}
    0 & 1   \\  -1 & 0
  \end{bmatrix},\;\;
  Q(x)=\begin{bmatrix}
    q(x) & p(x)   \\  p(x) & -q(x)
  \end{bmatrix},\;\;
   y(x)=\begin{bmatrix}
    y_{1}(x) \\ y_{2}(x)
  \end{bmatrix}, \nonumber
\end{align}
while $q(x)$ and $p(x)$ are complex-valued functions belong to $L_{2}(0,\pi)$, and $Q(x)=0$ on $(0,a)$. For any fixed pair of $\nu, j\in\{1,2\}$, denote by $B_{\nu,j}(Q)$ the boundary value problem for equation (\ref{1.1}) with the boundary conditions
\begin{align}
  y_{\nu}(0)=0,\quad   y_{j}(\pi)=0. \nonumber
\end{align}

In \cite{Buter2}, authors restrict themselves to the case $a\in[\frac{\pi}{2}, \pi)$, when the dependence of the characteristic functions of the problems $B_{\nu,j}(Q)$ on $Q(x)$ is linear. For the considered case, however, they give answers to the full range of questions usually raised in the inverse spectral theory. Specifically, reconstruction of two complex potentials $q$ and $p$ is studied from either complete spectra or subspectra of two problems $(B_{1,1}(Q), B_{1,2}(Q))$. They give conditions on the subspectra that are necessary and sufficient for the unique determination of the potentials. Moreover, necessary and sufficient conditions for the solvability of both inverse problems are obtained. For the inverse problem of recovering from the complete spectra, they establish also uniform stability in each ball of a finite radius.

In \cite{D-V0}, authors restrict themselves to the case $a\in[\frac{\pi}{3},\frac{\pi}{2})$, when the dependence of the characteristic functions of the problems $B_{\nu,j}(Q)$ on $Q(x)$ is nonlinear. They prove that complete spectra of two  problems $(B_{1,1}(Q), B_{1,2}(Q))$ can uniquely determine the two complex potentials $q$ and $p$ in the case $a\in[\frac{2\pi}{5},\frac{\pi}{2})$. Moreover, they provide a counterexample to illustrate that it is impossible in the case $a\in[\frac{\pi}{3},\frac{2\pi}{5})$.

In this paper, we restrict ourselves to the case $a\in[\frac{2\pi}{5},\frac{\pi}{2})$. For our considered case, we study the inverse problems of restoring two complex potentials $q$ and $p$ from complete spectra of two problems $(B_{1,1}(Q), B_{1,2}(Q))$ and $(B_{2,1}(Q), B_{2,2}(Q))$, respectively. We give full answers for our considered inverse problems: uniqueness, solvability and uniform stability.

Additionally, in the classical case $a=0$, inverse problems for (\ref{1.1}) were studied in [1, 16-19, 21-23] and other works. Meanwhile, it is worth mentioning that the first attempt of defining operator (\ref{0.1}) on a star-type graph was made in \cite{W-Y1, W-Y2}, which could be classified as locally nonlocal because the delay on each edge does not affect the other edges. Afterwards, Buterin suggests another concept of operator (\ref{0.1}) with delay on graphs that can be characterized as globally nonlocal, when the delay extends through vertices of the graph (for details, please refer to \cite{Buter}).

The paper is organized as follows. In the next section, we provide specific formulations for two inverse problems and obtain the characteristic functions of the boundary value problems $B_{\nu,j}(Q)$ as well as asymptotic formulae of eigenvalues. Section 3 is devoted to proving the uniqueness of the second inverse problem. In Section 4, we give necessary and sufficient conditions for the solvability of both inverse problems and reconstruction algorithms. In Section 5, the uniform stability of both inverse problems is studied.

\section{problem statement and characteristic function}
First, let us give the following theorem about the asymptotic relations of eigenvalues for the boundary value problems $B_{\nu,j}(Q)$, $\nu, j=1,2$.

\begin{theorem}\label{th1}
For $\nu, j\in\{1,2\}$, the boundary value problem $B_{\nu,j}(Q)$ has infinitely many eigenvalues $\lambda_{n,\nu,j}$, $n\in\mathbb{Z}$, of the form
\begin{align}\label{1.3}
\lambda_{n,\nu,j}=n+\frac{2-\nu-j}{2}+\kappa_{n,\nu,j},\quad \{\kappa_{n,\nu,j}\}_{n\in\mathbb{Z}}\in l_{2}.
\end{align}
\end{theorem}

Note that Theorem \ref{th1} has been proved in \cite{D-V0} for $\nu=1$. The proof of Theorem \ref{th1} for $\nu=2$ will be provided at the end of this section.

Assuming that the delay constant $a\in[\frac{2\pi}{5},\pi)$ is known a priori, we consider the following two inverse problems.

\textbf{Inverse Problem 1.} Given the two spectra $\{\lambda_{n,1,j}\}_{n\in\mathbb{Z}}$, $j=1,2$, find the potential functions $q$ and $p$.

\textbf{Inverse Problem 2.} Given the two spectra $\{\lambda_{n,2,j}\}_{n\in\mathbb{Z}}$, $j=1,2$, find the potential functions $q$ and $p$.

Next, we analyze the characteristic functions of the problems $B_{\nu,j}(Q)$, $\nu, j\!=\!1,2$. Let $Y(x,\lambda)$ be the fundamental matrix-solution of equation (\ref{1.1}) such that
\begin{align}
 Y(x,\lambda)=\begin{bmatrix}
    y_{1,1}(x,\lambda) & y_{1,2}(x,\lambda)   \\  y_{2,1}(x,\lambda) & y_{2,2}(x,\lambda)
  \end{bmatrix},\;\;
 Y(0,\lambda)=\begin{bmatrix}
    1 & 0   \\  0 & 1
  \end{bmatrix}. \nonumber
\end{align}
Then, for $\nu, j\in\{1,2\}$, eigenvalues of the problem $B_{\nu,j}(Q)$ coincide with zeros of the entire function
\begin{align}\label{2.1}
  \Delta_{\nu,j}(\lambda)=y_{j,3-\nu}(\pi,\lambda),
\end{align}
which is called characteristic function of $B_{\nu,j}(Q)$.

When $a\in[\frac{\pi}{3},\frac{\pi}{2}$), it follows from the relation (12) in \cite{Buter2} that
\begin{align}\label{2.3}
  Y(x,\lambda)=Y_{0}(x,\lambda)+Y_{1}(x,\lambda)+Y_{2}(x,\lambda),\qquad\quad
\end{align}
where
\begin{align}\label{2.4}
  Y_{0}(x,\lambda)=\begin{bmatrix}
   \cos\lambda x & -\sin\lambda x   \\  \sin\lambda x  & \cos\lambda x
  \end{bmatrix},\qquad\qquad\qquad\quad\,
\end{align}
\begin{align}\label{2.5}
  Y_{1}(x,\lambda)=B\int_{a}^{x}Y_{0}(x-t,\lambda)Q(t)Y_{0}(t-a,\lambda)dt,
\end{align}
\begin{align}\label{2.6}
  Y_{2}(x,\lambda)=B\int_{2a}^{x}Y_{0}(x-t,\lambda)Q(t)Y_{1}(t-a,\lambda)dt.
\end{align}

Combining these relations (\ref{2.3})-(\ref{2.6}) and taking (\ref{2.1}) into account, by a rather tedious computation, we obtain

\begin{align}
 \Delta_{1,1}(\lambda)\!=\!-\sin\!\lambda\pi\!-\!\!\int_{a}^{\pi}\!\!q(t)\cos\!\lambda(\pi\!-\!2t\!+\!a)dt
 \!+\!\!\int_{a}^{\pi}\!\!p(t)\sin\!\lambda(\pi\!-\!2t\!+\!a)dt \nonumber\\
 -\int_{2a}^{\pi}dt\int_{a}^{t-a}\Big(q(t)q(s)+p(t)p(s)\Big)\sin\!\lambda(\pi\!-\!2t\!+\!2s)ds \qquad\;\,\nonumber\\
 +\int_{2a}^{\pi}dt\int_{a}^{t-a}\Big(q(t)p(s)-p(t)q(s)\Big)\cos\!\lambda(\pi\!-\!2t\!+\!2s)ds, \qquad\nonumber
\end{align}

\begin{align}
 \Delta_{1,2}(\lambda)\!=\!\cos\!\lambda\pi\!-\!\!\int_{a}^{\pi}\!\!q(t)\sin\!\lambda(\pi\!-\!2t\!+\!a)dt
 \!-\!\!\int_{a}^{\pi}\!\!p(t)\cos\!\lambda(\pi\!-\!2t\!+\!a)dt \nonumber \;\;\\
 +\int_{2a}^{\pi}dt\int_{a}^{t-a}\Big(q(t)p(s)-p(t)q(s)\Big)\sin\!\lambda(\pi\!-\!2t\!+\!2s)ds \qquad\;\nonumber\\
 +\int_{2a}^{\pi}dt\int_{a}^{t-a}\Big(q(t)q(s)+p(t)p(s)\Big)\cos\!\lambda(\pi\!-\!2t\!+\!2s)ds, \qquad\!\nonumber
\end{align}

\begin{align}
 \Delta_{2,1}(\lambda)\!=\!\cos\!\lambda\pi\!+\!\!\int_{a}^{\pi}\!\!q(t)\sin\!\lambda(\pi\!-\!2t\!+\!a)dt
 \!+\!\!\int_{a}^{\pi}\!\!p(t)\cos\!\lambda(\pi\!-\!2t\!+\!a)dt \nonumber\\
 +\int_{2a}^{\pi}dt\int_{a}^{t-a}\Big(q(t)p(s)-p(t)q(s)\Big)\sin\!\lambda(\pi\!-\!2t\!+\!2s)ds \qquad\!\nonumber\\
 +\int_{2a}^{\pi}dt\int_{a}^{t-a}\Big(q(t)q(s)+p(t)p(s)\Big)\cos\!\lambda(\pi\!-\!2t\!+\!2s)ds, \quad\;\:\nonumber
\end{align}

\begin{align}
 \Delta_{2,2}(\lambda)\!=\!\sin\!\lambda\pi\!+\!\!\int_{a}^{\pi}\!\!p(t)\sin\!\lambda(\pi\!-\!2t\!+\!a)dt
 \!-\!\!\int_{a}^{\pi}\!\!q(t)\cos\!\lambda(\pi\!-\!2t\!+\!a)dt \nonumber \\
 +\int_{2a}^{\pi}dt\int_{a}^{t-a}\Big(q(t)q(s)+p(t)p(s)\Big)\sin\!\lambda(\pi\!-\!2t\!+\!2s)ds \qquad\!\nonumber\\
 -\int_{2a}^{\pi}dt\int_{a}^{t-a}\Big(q(t)p(s)-p(t)q(s)\Big)\cos\!\lambda(\pi\!-\!2t\!+\!2s)ds. \quad\:\:\nonumber
\end{align}
Changing the variable and interchanging the order of integration, we obtain
\begin{align}\label{2.7.1}
 \Delta_{1,1}(\lambda)\!=\!-\sin\!\lambda\pi\!+\!\!\int_{a-\pi}^{\pi-a}\!\!v_{1,1}(x)\sin\!\lambda xdx\!
 +\!\!\int_{a-\pi}^{\pi-a}\!\!v_{1,2}(x)\cos\!\lambda xdx,\!\!\!
\end{align}

\begin{align}\label{2.8.1}
 \Delta_{1,2}(\lambda)\!=\!\cos\!\lambda\pi\!+\!\!\int_{a-\pi}^{\pi-a}\!\!v_{1,2}(x)\sin\!\lambda xdx
 \!-\!\int_{a-\pi}^{\pi-a}\!\!v_{1,1}(x)\cos\!\lambda xdx,\;
\end{align}

\begin{align}\label{2.7}
 \Delta_{2,1}(\lambda)\!=\!\cos\!\lambda\pi\!+\!\!\int_{a-\pi}^{\pi-a}\!\!v_{2,1}(x)\sin\!\lambda xdx\!
 +\!\!\int_{a-\pi}^{\pi-a}\!\!v_{2,2}(x)\cos\!\lambda xdx,\,
\end{align}

\begin{align}\label{2.8}
 \Delta_{2,2}(\lambda)\!=\!\sin\!\lambda\pi\!+\!\!\int_{a-\pi}^{\pi-a}\!\!v_{2,2}(x)\sin\!\lambda xdx
 \!-\!\int_{a-\pi}^{\pi-a}\!\!v_{2,1}(x)\cos\!\lambda xdx,
\end{align}
where
\begin{align}\label{2.9.1}
v_{1,1}(x)\!=\!\begin{cases}
\frac{1}{2}p(\frac{\pi+a-x}{2})\!-\!\frac{1}{2}\int_{\frac{\pi+2a-x}{2}}^{\pi}\Big[q(t)q(\frac{x+2t-\pi}{2})  \\
\qquad\quad\quad+p(t)p(\frac{x+2t-\pi}{2})\Big]dt, \;\: x\!\in\!(2a\!-\!\pi,\pi\!-\!2a),\\
\\
\frac{1}{2}p(\frac{\pi+a-x}{2}),\;\:x\!\in[a\!-\!\pi,2a\!-\!\pi]\cup[\pi\!-\!2a,\pi\!-\!a],
\end{cases}\!\!\!\!\!\!\!\!\!\!
\end{align}

\begin{align}\label{2.10.1}
v_{1,2}(x)\!=\!\begin{cases}
-\frac{1}{2}q(\frac{\pi+a-x}{2})\!+\!\frac{1}{2}\int_{\frac{\pi+2a-x}{2}}^{\pi}\Big[q(t)p(\frac{x+2t-\pi}{2})  \\
\qquad\quad\quad-p(t)q(\frac{x+2t-\pi}{2})\Big]dt, \;\: x\!\in\!(2a\!-\!\pi,\pi\!-\!2a),\\
\\
-\frac{1}{2}q(\frac{\pi+a-x}{2}),\;\:x\!\in[a\!-\!\pi,2a\!-\!\pi]\cup[\pi\!-\!2a,\pi\!-\!a],
\end{cases}\!\!\!\!\!\!\!\!\!\!
\end{align}

\begin{align}\label{2.9}
v_{2,1}(x)\!=\!\begin{cases}
\frac{1}{2}q(\frac{\pi+a-x}{2})\!+\!\frac{1}{2}\int_{\frac{\pi+2a-x}{2}}^{\pi}\Big[q(t)p(\frac{x+2t-\pi}{2})  \\
\qquad\quad\quad-p(t)q(\frac{x+2t-\pi}{2})\Big]dt, \;\: x\!\in\!(2a\!-\!\pi,\pi\!-\!2a),\\
\\
\frac{1}{2}q(\frac{\pi+a-x}{2}),\;\:x\!\in[a\!-\!\pi,2a\!-\!\pi]\cup[\pi\!-\!2a,\pi\!-\!a],
\end{cases}\!\!\!\!\!\!\!\!\!\!
\end{align}

\begin{align}\label{2.10}
v_{2,2}(x)\!=\!\begin{cases}
\frac{1}{2}p(\frac{\pi+a-x}{2})\!+\!\frac{1}{2}\int_{\frac{\pi+2a-x}{2}}^{\pi}\Big[q(t)q(\frac{x+2t-\pi}{2})  \\
\qquad\quad\quad+p(t)p(\frac{x+2t-\pi}{2})\Big]dt, \;\: x\!\in\!(2a\!-\!\pi,\pi\!-\!2a),\\
\\
\frac{1}{2}p(\frac{\pi+a-x}{2}),\;\:x\!\in[a\!-\!\pi,2a\!-\!\pi]\cup[\pi\!-\!2a,\pi\!-\!a].
\end{cases}\!\!\!\!\!\!\!\!\!\!
\end{align}

Using Euler's formula, the relations (\ref{2.7.1}) and (\ref{2.8.1}) take the forms
\begin{align}\label{2.11.1}
 \Delta_{1,1}(\lambda)=-\sin\lambda\pi+\int_{a-\pi}^{\pi-a}u_{1,1}(x)\exp(i\lambda x)dx,\!\!\!\!\!
\end{align}

\begin{align}\label{2.12.1}
 \Delta_{1,2}(\lambda)=\cos\lambda\pi+\int_{a-\pi}^{\pi-a}u_{1,2}(x)\exp(i\lambda x)dx,
\end{align}
and the relations (\ref{2.7}) and (\ref{2.8}) take the forms
\begin{align}\label{2.11}
 \Delta_{2,1}(\lambda)=\cos\lambda\pi+\int_{a-\pi}^{\pi-a}u_{2,1}(x)\exp(i\lambda x)dx,
\end{align}

\begin{align}\label{2.12}
 \Delta_{2,2}(\lambda)=\sin\lambda\pi+\int_{a-\pi}^{\pi-a}u_{2,2}(x)\exp(i\lambda x)dx,
\end{align}
where
\begin{align}\label{2.13.1}
 u_{1,1}(x)=\frac{v_{1,1}(x)-v_{1,1}(-x)}{2i}+\frac{v_{1,2}(x)+v_{1,2}(-x)}{2},
\end{align}
\begin{align}\label{2.14.1}
 u_{1,2}(x)=\frac{v_{1,2}(x)-v_{1,2}(-x)}{2i}-\frac{v_{1,1}(x)+v_{1,1}(-x)}{2},
\end{align}
\begin{align}\label{2.13}
 u_{2,1}(x)=\frac{v_{2,1}(x)-v_{2,1}(-x)}{2i}+\frac{v_{2,2}(x)+v_{2,2}(-x)}{2},
\end{align}
\begin{align}\label{2.14}
 u_{2,2}(x)=\frac{v_{2,2}(x)-v_{2,2}(-x)}{2i}-\frac{v_{2,1}(x)+v_{2,1}(-x)}{2}.
\end{align}

By the standard approach involving Rouch\'{e}'s theorem, one can prove the following assertion, which can also be obtained as a direct corollary from Theorem 4 in \cite{Buter1}.

\begin{lemma}\label{lem1}
For $j\in\{1,2\}$, the function $\Delta_{2,j}(\lambda)$  has infinitely many zeros $\lambda_{n,2,j}$, $n\in\mathbb{Z}$, of the form (\ref{1.3}) for $\nu=2$.
\end{lemma}

We note that Lemma \ref{lem1} immediately implies the assertion of Theorem  \ref{th1} for $\nu=2$..

\section{uniqueness of inverse problem}
	
	As mentioned in the introduction, the uniqueness of Inverse Problem 1 has been prove in \cite{D-V0}. Next, we give a uniqueness theorem for Inverse Problem 2. To this end, for $\nu,j=1,2$, along with the problem $B_{\nu,j}(Q)$, we will consider other problem $B_{\nu,j}(\widetilde{Q})$ of the same form but with a different potential matrix
\begin{equation}
  \widetilde{Q}(x)=\begin{bmatrix}
    \widetilde{q}(x) & \widetilde{p}(x)   \\  \widetilde{p}(x) & -\widetilde{q}(x)
  \end{bmatrix}. \nonumber
\end{equation}
We agree that if a certain symbol $\alpha$ denotes an object related to the problem $B_{\nu,j}(Q)$, then  this symbol with tilde $\widetilde{\alpha}$ will denote the analogous object related to the problem $B_{\nu,j}(\widetilde{Q})$ .

\begin{theorem}\label{th2}
 If $\lambda_{n,2,1}=\widetilde{\lambda}_{n,2,1}$, $\lambda_{n,2,2}=\widetilde{\lambda}_{n,2,2}$, $n\in\mathbb{Z}$, then $q(x)=\widetilde{q}(x)$ and $p(x)=\widetilde{p}(x)$ a.e. on $[a,\pi]$. That is, the specification of two spectra $\{\lambda_{n,2,j}\}_{n\in\mathbb{Z}}$, $j=1,2$, uniquely determines the potentials $q$ and $p$.
\end{theorem}

Before proceeding directly to the proof of Theorem \ref{th2}, we fulfil some preparatory work. Let
\begin{align}\label{3.1}
  w_{2,1}(x)\!=\!(iu_{2,1}\!-\!u_{2,2})(\pi\!+\!a\!-\!2x)\!-\!(iu_{2,1}\!+\!u_{2,2})(2x\!-\!\pi\!-\!a),
\end{align}
\begin{align}\label{3.2}
  w_{2,2}(x)\!=\!(u_{2,1}\!+\!iu_{2,2})(\pi\!+\!a\!-\!2x)\!+\!(u_{2,1}\!-\!iu_{2,2})(2x\!-\!\pi\!-\!a),
\end{align}
then the functions $u_{2,1}(x)$ and $u_{2,2}(x)$ in $L_{2}(a-\pi,\pi-a)$ uniquely determine the functions $w_{2,1}(x)$ and $w_{2,2}(x)$ in $L_{2}(a,\pi)$.  In addition, the relations (\ref{3.1}) and (\ref{3.2}) imply the estimates
\begin{align}\label{3.2-2}
\begin{cases}
  \|w_{2,1}\|_{L_{2}(a,\pi)}\leq\!2\sqrt{2}\left(\|u_{2,1}\|_{L_{2}(a-\pi,\pi-a)}\!+\!\|u_{2,2}\|_{L_{2}(a-\pi,\pi-a)}\right),\\
  \|w_{2,2}\|_{L_{2}(a,\pi)}\leq\!2\sqrt{2}\left(\|u_{2,1}\|_{L_{2}(a-\pi,\pi-a)}\!+\!\|u_{2,2}\|_{L_{2}(a-\pi,\pi-a)}\right).
\end{cases}
\end{align}

When $x\!\in\![a\!-\!\pi,2a\!-\!\pi]\cup[\pi\!-\!2a,\pi\!-\!a]$, according to (\ref{2.9})-(\ref{2.10}) and (\ref{2.13})-(\ref{2.14}), one can calculate \begin{align}
  2(u_{2,1}(x)+iu_{2,2}(x))=(p\!-\!iq)\!\left(\frac{\pi\!+\!a\!-\!x}{2}\right),\nonumber
\end{align}
\begin{align}
  2(u_{2,1}(x)-iu_{2,2}(x))=(p\!+\!iq)\!\left(\frac{\pi\!+\!a\!+\!x}{2}\right).\nonumber
\end{align}
The changes of variables $t=\frac{\pi+a-x}{2}$ and $t=\frac{\pi+a+x}{2}$, respectively, lead to
\begin{align}
(p\!-\!iq)(t)=2(u_{2,1}\!+\!iu_{2,2})(\pi\!+\!a\!-\!2t),\quad t\in[a,\frac{3a}{2}]\cup[\pi-\frac{a}{2},\pi],\nonumber
\end{align}
\begin{align}
(p\!+\!iq)(t)=2(u_{2,1}\!-\!iu_{2,2})(2t\!-\!\pi\!-\!a),\quad t\in[a,\frac{3a}{2}]\cup[\pi-\frac{a}{2},\pi].\nonumber
\end{align}
Summing up two equations above and subtracting one from the other, and taking (\ref{3.1})-(\ref{3.2}) into account, we get
\begin{align}
  p(t)=w_{2,2}(t),\quad q(t)=w_{2,1}(t),\quad t\in[a,\frac{3a}{2}]\cup[\pi-\frac{a}{2},\pi].\nonumber
\end{align}

We use the symbol $f|_{\mathcal{S}}$ for denoting the restriction of the function $f$ to the set $\mathcal{S}$. Thus, we have proved the following lemma.

\begin{lemma}\label{lem2}
The following relations hold:
\begin{equation}\label{3.3}
\begin{cases}
  q|_{[a,\frac{3a}{2}]\cup[\pi-\frac{a}{2},\pi]}=w_{2,1}|_{[a,\frac{3a}{2}]\cup[\pi-\frac{a}{2},\pi]},\\
  p|_{[a,\frac{3a}{2}]\cup[\pi-\frac{a}{2},\pi]}=w_{2,2}|_{[a,\frac{3a}{2}]\cup[\pi-\frac{a}{2},\pi]}.
\end{cases}
\end{equation}
\end{lemma}

When $x\!\in(2a\!-\!\pi,\pi\!-\!2a)$, according to (\ref{2.9})-(\ref{2.10}) and (\ref{2.13})-(\ref{2.14}), one can calculate
\begin{align}
  2(u_{2,1}(x)+iu_{2,2}(x))=(p\!-\!iq)\!\left(\frac{\pi\!+\!a\!-\!x}{2}\right) \qquad\qquad\qquad\qquad\qquad\nonumber\\
  +\!\!\int_{\frac{\pi+2a-x}{2}}^{\pi}\!(q(t)\!+\!ip(t))(q\!-\!ip)\!\!\left(\frac{x\!+\!2t\!-\!\pi}{2}\right)dt,\!\!\nonumber
\end{align}
\begin{align}
  2(u_{2,1}(x)-iu_{2,2}(x))=(p\!+\!iq)\!\left(\frac{\pi\!+\!a\!+\!x}{2}\right)\qquad\qquad\qquad\qquad\qquad\nonumber\\
  +\!\!\int_{\frac{\pi+2a+x}{2}}^{\pi}\!(q(t)\!-\!ip(t))(q\!+\!ip)\!\!\left(\frac{2t\!-\!x\!-\!\pi}{2}\right)dt.\!\!\nonumber
\end{align}
The changes of variables $\xi=\frac{\pi+a-x}{2}$ and $\xi=\frac{\pi+a+x}{2}$, respectively, lead to
\begin{align}
  2(u_{2,1}+iu_{2,2})(\pi+a-2\xi)\!=\!(p\!-\!iq)(\xi) \qquad\qquad\qquad\qquad\qquad\;\;\;\nonumber\\
  +\!\int_{\xi+\frac{a}{2}}^{\pi}\!(q(t)\!+\!ip(t))(q\!-\!ip)\!\!\left(t\!-\!\xi\!+\!\frac{a}{2}\right)\!dt,\!\!\!\!\!\!\!\nonumber
\end{align}
\begin{align}
  2(u_{2,1}-iu_{2,2})(2\xi-\pi-a)\!=\!(p\!+\!iq)(\xi) \qquad\qquad\qquad\qquad\qquad\;\;\;\nonumber\\
  +\!\int_{\xi+\frac{a}{2}}^{\pi}\!(q(t)\!-\!ip(t))(q\!+\!ip)\!\!\left(t\!-\!\xi\!+\!\frac{a}{2}\right)\!dt\!\!\!\!\!\!\nonumber
\end{align}
for $\xi\in(\frac{3a}{2},\pi-\frac{a}{2})$. Summing up two equations above and subtracting one from the other, and taking (\ref{3.1})-(\ref{3.2}) into account, we get
\begin{align}\label{3.4}
 w_{2,2}(\xi)=p(\xi)\!+\!\int_{\xi+\frac{a}{2}}^{\pi}\left[q(t)q\!\left(t\!-\!\xi\!+\!\frac{a}{2}\right)\!+\!p(t)p\!\left(t\!-\!\xi\!
  +\!\frac{a}{2}\right)\right]dt,
\end{align}
\begin{align}\label{3.5}
 w_{2,1}(\xi)=q(\xi)\!+\!\int_{\xi+\frac{a}{2}}^{\pi}\left[q(t)p\!\left(t\!-\!\xi\!+\!\frac{a}{2}\right)\!-\!p(t)q\!\left(t\!-\!\xi\!
  +\!\frac{a}{2}\right)\right]dt
\end{align}
for $\xi\in(\frac{3a}{2},\pi-\frac{a}{2})$.

Since $a\in[\frac{2\pi}{5},\frac{\pi}{2})$, it is easy to see that
\begin{equation}
  a<\pi-a\leq\frac{3a}{2}<\pi-\frac{a}{2}\leq2a<\pi. \nonumber
\end{equation}
Moreover, when $\xi\in(\frac{3a}{2},\pi-\frac{a}{2})$, it is easy to find that
\begin{equation}
  2a<\xi+\frac{a}{2}\leq t\leq\pi,\quad a\leq t-\xi+\frac{a}{2}<\pi-a. \nonumber
\end{equation}
So, according to Lemma \ref{lem2}, the relations (\ref{3.4})
and (\ref{3.5}) yield
\begin{align}
 p(\xi)=w_{2,2}(\xi)\!-\!\int_{\xi+\frac{a}{2}}^{\pi}\left[w_{2,1}(t)w_{2,1}\!\left(t\!-\!\xi\!+\!\frac{a}{2}\right)\!+\!w_{2,2}(t)w_{2,2}
  \!\left(t\!-\!\xi\!+\!\frac{a}{2}\right)\right]dt, \nonumber
\end{align}
\begin{align}
   q(\xi)=w_{2,1}(\xi)\!-\!\int_{\xi+\frac{a}{2}}^{\pi}\left[w_{2,1}(t)w_{2,2}\!\left(t\!-\!\xi\!+\!\frac{a}{2}\right)\!-\!w_{2,2}(t)w_{2,1}
  \!\left(t\!-\!\xi\!+\!\frac{a}{2}\right)\right]dt \nonumber
\end{align}
for $\xi\in(\frac{3a}{2},\pi-\frac{a}{2})$.

Thus, we have proved the following lemma.

\begin{lemma}\label{lem3}
The following relations hold:
\begin{align}\label{3.6}
\begin{cases}
  q|_{(\frac{3a}{2},\pi-\frac{a}{2})}=(w_{2,1}-\gamma_{2,1})|_{(\frac{3a}{2},\pi-\frac{a}{2})},\\
  p|_{(\frac{3a}{2},\pi-\frac{a}{2})}=(w_{2,2}-\gamma_{2,2})|_{(\frac{3a}{2},\pi-\frac{a}{2})},
\end{cases}
\end{align}
where
\begin{align}\label{3.7}
\begin{cases}
\gamma_{2,1}(x)\!\!=\!\!\int_{x+\frac{a}{2}}^{\pi}\!\left[w_{2,1}(t)w_{2,2}\!\left(t\!-\!x\!+\!\frac{a}{2}\right)\!-\!w_{2,2}(t)w_{2,1}\!\left(t\!-\!x\!
  +\!\frac{a}{2}\right)\right]\!dt,\\
\gamma_{2,2}(x)\!\!=\!\!\int_{x+\frac{a}{2}}^{\pi}\!\left[w_{2,1}(t)w_{2,1}\!\left(t\!-\!x\!+\!\frac{a}{2}\right)\!+\!w_{2,2}(t)w_{2,2}\!\left(t\!-\!x\!
  +\!\frac{a}{2}\right)\right]\!dt
\end{cases}\!\!\!\!\!\!
\end{align}
for $x\in(\frac{3a}{2},\pi-\frac{a}{2})$.
\end{lemma}

Note that the functions $\Delta_{2,1}(\lambda)$ and $\Delta_{2,2}(\lambda)$ take the forms (\ref{2.11}) and (\ref{2.12}), respectively. Hence, we have the following lemma, which is as direct corollary of the corresponding general assertions in \cite{Buter1}.

\begin{lemma}\label{lem4}
The functions $\Delta_{2,1}(\lambda)$ and $\Delta_{2,2}(\lambda)$ are uniquely determined by specifying their zeros. Moreover, the following  representations hold:
\begin{align}\label{1.4}
  \Delta_{2,1}(\lambda)=\prod_{n\in\mathbb{Z}}\frac{\lambda_{n,2,1}-\lambda}{n-\frac{1}{2}}\exp\left(\frac{\lambda}{n-\frac{1}{2}}\right),
  \qquad\qquad\quad\;\;
\end{align}
\begin{align}\label{1.4.1}
\Delta_{2,2}(\lambda)=\pi(\lambda-\lambda_{1,2,2})\!\!\prod_{n\in\mathbb{Z},n\neq1}\!\!
\frac{\lambda_{n,2,2}-\lambda}{n-1}\exp\left(\frac{\lambda}{n-1}\right).\!\!\!\!\!\!\!\!\!\!\!\!\!\!\!
\end{align}
\end{lemma}

\begin{proof}
Theorem 5 in \cite{Buter1} gives the representation holds:
\begin{align}\label{1.4-1}
  \Delta_{2,1}(\lambda)=\alpha\exp(\beta\lambda)\prod_{n\in\mathbb{Z}}\frac{\lambda_{n,2,1}-\lambda}{n-\frac{1}{2}}\exp
  \left(\frac{\lambda}{n-\frac{1}{2}}\right),
\end{align}
where $\beta=\gamma$ and
\begin{equation}
 \alpha=\lim_{\lambda\rightarrow 0}\cos\lambda\pi=1, \quad \gamma=\lim_{\lambda\rightarrow 0}\frac{d}{d\lambda}\ln(\cos\lambda\pi)=0.\nonumber
\end{equation}
Hence, the formula (\ref{1.4-1}) takes the form as (\ref{1.4}).
As well as Theorem 5 in \cite{Buter1} gives the representation holds:
\begin{align}\label{1.4-2}
  \Delta_{2,2}(\lambda)\!=\!\alpha_{1}\!\exp((\beta_{1}\!\!-\!\!1)\lambda)(\lambda\!-\!\lambda_{1,2,2})\!\!\!\prod_{n\in\mathbb{Z},n\neq1}
  \!\!\!\!\frac{\lambda_{n,2,2}\!-\!\lambda}{n-1}\exp\!\left(\!\frac{\lambda}{n\!-\!1}\!\right),
\end{align}
where $\beta_{1}=1+\gamma_{1}$ and
\begin{equation}
 \alpha_{1}=\lim_{\lambda\rightarrow 0}\frac{\sin\lambda\pi}{\lambda}=\pi, \quad \gamma_{1}=\lim_{\lambda\rightarrow 0}\frac{d}{d\lambda}\ln(\frac{\sin\lambda\pi}{\lambda})=0.\nonumber
\end{equation}
Hence, the formula (\ref{1.4-2}) takes the form as (\ref{1.4.1}).
\end{proof}

Now we are in position to give the proof of Theorem \ref{th2}.

\textbf{Proof of Theorem \ref{th2}.} Since $\lambda_{n,2,1}=\widetilde{\lambda}_{n,2,1}$, $\lambda_{n,2,2}=\widetilde{\lambda}_{n,2,2}$, $n\in\mathbb{Z}$, according to Lemma \ref{lem4}, we have
\begin{equation}
  \Delta_{2,1}(\lambda)=\widetilde{\Delta}_{2,1}(\lambda),\quad \Delta_{2,2}(\lambda)=\widetilde{\Delta}_{2,2}(\lambda),\nonumber
\end{equation}
which yields
\begin{equation}
  \Delta_{2,1}(\lambda)\!-\!\cos\lambda\pi\!=\!\widetilde{\Delta}_{2,1}(\lambda)\!-\!\cos\lambda\pi,\; \Delta_{2,2}(\lambda)\!-\!\sin\lambda\pi\!=\!\widetilde{\Delta}_{2,2}(\lambda)\!-\!\sin\lambda\pi.\nonumber
\end{equation}
From (\ref{2.11}) and (\ref{2.12}), one has
\begin{align}
  \int_{a-\pi}^{\pi-a}u_{2,1}(x)\exp(i\lambda x)dx=\int_{a-\pi}^{\pi-a}\widetilde{u}_{2,1}(x)\exp(i\lambda x)dx,\nonumber\\
  \int_{a-\pi}^{\pi-a}u_{2,2}(x)\exp(i\lambda x)dx=\int_{a-\pi}^{\pi-a}\widetilde{u}_{2,2}(x)\exp(i\lambda x)dx.\nonumber
\end{align}
By the arbitrariness of $\lambda\in\mathbb{C}$, the above two equations imply
\begin{align}\label{3.8}
  u_{2,1}(x)=\widetilde{u}_{2,1}(x),\; u_{2,2}(x)=\widetilde{u}_{2,2}(x),\; a.e. \;on \;[a-\pi,\pi-a].
\end{align}

In view of (\ref{3.1}), (\ref{3.2}) and (\ref{3.8}), we get
\begin{align}\label{3.9}
  w_{2,1}(x)=\widetilde{w}_{2,1}(x),\quad w_{2,2}(x)=\widetilde{w}_{2,2}(x),\quad a.e. \;on \;[a,\pi].
\end{align}
According to Lemmas \ref{lem2} and \ref{lem3}, it follows from (\ref{3.9}) that
\begin{align}
  q(x)=\widetilde{q}(x),\quad p(x)=\widetilde{p}(x),\quad a.e. \;on \;[a,\pi].\nonumber
\end{align}
The proof of Theorem \ref{th2} is finished. $\square$

Let
\begin{align}\label{3.1.1}
  w_{1,1}(x)\!=\!-(u_{1,1}\!+\!iu_{1,2})(\pi\!+\!a\!-\!2x)\!-\!(u_{1,1}\!-\!iu_{1,2})(2x\!-\!\pi\!-\!a),\!\!\!\!
\end{align}
\begin{align}\label{3.2.1}
  w_{1,2}(x)\!=\!(iu_{1,1}\!-\!u_{1,2})(\pi\!+\!a\!-\!2x)\!-\!(iu_{1,1}\!+\!u_{1,2})(2x\!-\!\pi\!-\!a),
\end{align}
then the functions $u_{1,1}(x)$ and $u_{1,2}(x)$ in $L_{2}(a-\pi,\pi-a)$ uniquely determine the functions $w_{1,1}(x)$ and $w_{1,2}(x)$ in $L_{2}(a,\pi)$.  In addition, the relations (\ref{3.1.1}) and (\ref{3.2.1}) imply the estimates
\begin{align}\label{3.2-1}
\begin{cases}
  \|w_{1,1}\|_{L_{2}(a,\pi)}\leq\!2\sqrt{2}\left(\|u_{1,1}\|_{L_{2}(a-\pi,\pi-a)}\!+\!\|u_{1,2}\|_{L_{2}(a-\pi,\pi-a)}\right),\\
  \|w_{1,2}\|_{L_{2}(a,\pi)}\leq\!2\sqrt{2}\left(\|u_{1,1}\|_{L_{2}(a-\pi,\pi-a)}\!+\!\|u_{1,2}\|_{L_{2}(a-\pi,\pi-a)}\right).
\end{cases}
\end{align}

Similar to the above discussion, we can obtain the following lemmas, which are necessary for studying the solvability and stability of Inverse Problem 1.

\begin{lemma}\label{lem2.1}
The following relations hold:
\begin{equation}\label{3.3.1}
\begin{cases}
  q|_{[a,\frac{3a}{2}]\cup[\pi-\frac{a}{2},\pi]}=w_{1,1}|_{[a,\frac{3a}{2}]\cup[\pi-\frac{a}{2},\pi]},\\
  p|_{[a,\frac{3a}{2}]\cup[\pi-\frac{a}{2},\pi]}=w_{1,2}|_{[a,\frac{3a}{2}]\cup[\pi-\frac{a}{2},\pi]}.
\end{cases}
\end{equation}
\end{lemma}

\begin{lemma}\label{lem3.1}
The following relations hold:
\begin{align}\label{3.6.1}
\begin{cases}
  q|_{(\frac{3a}{2},\pi-\frac{a}{2})}=(w_{1,1}+\gamma_{1,1})|_{(\frac{3a}{2},\pi-\frac{a}{2})},\\
  p|_{(\frac{3a}{2},\pi-\frac{a}{2})}=(w_{1,2}+\gamma_{1,2})|_{(\frac{3a}{2},\pi-\frac{a}{2})},
\end{cases}
\end{align}
where
\begin{align}\label{3.7.1}
\begin{cases}
  \gamma_{1,1}(x)\!\!=\!\!\int_{x+\frac{a}{2}}^{\pi}\!\left[w_{1,1}(t)w_{1,2}\!\left(t\!-\!x\!+\!\frac{a}{2}\right)\!-\!w_{1,2}(t)w_{1,1}\!\left(t\!-\!x\!
  +\!\frac{a}{2}\right)\right]\!dt,\!\!\\
  \gamma_{1,2}(x)\!\!=\!\!\int_{x+\frac{a}{2}}^{\pi}\!\left[w_{1,1}(t)w_{1,1}\!\left(t\!-\!x\!+\!\frac{a}{2}\right)\!+\!w_{1,2}(t)w_{1,2}\!\left(t\!-\!x\!
  +\!\frac{a}{2}\right)\right]\!dt
\end{cases}\!\!\!\!\!\!
\end{align}
for $x\in(\frac{3a}{2},\pi-\frac{a}{2})$.
\end{lemma}

\begin{lemma}\label{lem4.1}
The functions $\Delta_{1,1}(\lambda)$ and $\Delta_{1,2}(\lambda)$ are uniquely determined by specifying their zeros. Moreover, Moreover, the following  representations hold:
\begin{align}\label{1.4-1}
  \Delta_{1,1}(\lambda)=\pi(\lambda_{0,1,1}-\lambda)\prod_{|n|\in\mathbb{N}}\frac{\lambda_{n,1,1}-\lambda}{n}\exp\left(\frac{\lambda}{n}\right),
\end{align}
\begin{align}\label{1.4.1-1}
\Delta_{1,2}(\lambda)=\prod_{n\in\mathbb{Z}}\frac{\lambda_{n,1,2}-\lambda}{n-\frac{1}{2}}\exp\left(\frac{\lambda}{n-\frac{1}{2}}\right).\qquad\qquad
\end{align}
\end{lemma}

\section{solvability of inverse problem}

To begin with, let us give a necessary and sufficient condition for the solvability of Inverse Problems 1 and 2, respectively.

\begin{theorem}\label{th3.1}
For any sequences of complex numbers $\{\lambda_{n,1,1}\}_{n\in\mathbb{Z}}$ and $\{\lambda_{n,1,2}\}_{n\in\mathbb{Z}}$ to be the spectra of some boundary value problems $B_{1,1}(Q)$ and $B_{1,2}(Q)$ with $q,p\in L_{2}(0,\pi)$, respectively, it is necessary and sufficient to satisfy the following two conditions:

(i) For $j=1,2$, the sequence $\{\lambda_{n,1,j}\}_{n\in\mathbb{Z}}$ has the form (\ref{1.3}) for $\nu=1$;

(ii) The exponential types of the functions $\Delta_{1,1}(\lambda)\,+\,\sin\lambda\pi$ and $\Delta_{1,2}(\lambda)-\cos\lambda\pi$ do not exceed $\pi-a$, where the functions $\Delta_{1,1}(\lambda)$ and $\Delta_{1,2}(\lambda)$ are determined by the formulae (\ref{1.4-1}) and (\ref{1.4.1-1}), respectively.
\end{theorem}

\begin{theorem}\label{th3}
For any sequences of complex numbers $\{\lambda_{n,2,1}\}_{n\in\mathbb{Z}}$ and $\{\lambda_{n,2,2}\}_{n\in\mathbb{Z}}$ to be the spectra of some boundary value problems $B_{2,1}(Q)$ and $B_{2,2}(Q)$ with $q,p\in L_{2}(0,\pi)$, respectively, it is necessary and sufficient to satisfy the following two conditions:

(i) For $j=1,2$, the sequence $\{\lambda_{n,2,j}\}_{n\in\mathbb{Z}}$ has the form (\ref{1.3}) for $\nu=2$;

(ii) The exponential types of the functions $\Delta_{2,1}(\lambda)-\cos\lambda\pi$ and $\Delta_{2,2}(\lambda)-\sin\lambda\pi$ do not exceed $\pi-a$, where the functions $\Delta_{2,1}(\lambda)$ and $\Delta_{2,2}(\lambda)$ are determined by the formulae (\ref{1.4}) and (\ref{1.4.1}), respectively.
\end{theorem}

To prove Theorem \ref{th3.1}, we need the following lemma, which has been given (see Lemma 3 in \cite{Buter2}).

\begin{lemma}\label{lem5}
For any complex sequences $\{\lambda_{n,1,1}\}_{n\in\mathbb{Z}}$ and $\{\lambda_{n,1,2}\}_{n\in\mathbb{Z}}$ of the form (\ref{1.3}) for $\nu=1$,
the functions $\Delta_{1,1}(\lambda)$ and $\Delta_{1,2}(\lambda)$ constructed by the formulae in (\ref{1.4-1}) and (\ref{1.4.1-1}) have the following forms:
\begin{align}\label{4.1}
  \Delta_{1,1}(\lambda)=-\sin\lambda\pi+\int_{-\pi}^{\pi}u_{1,1}(x)\exp(i\lambda x)dx,
\end{align}
\begin{align}\label{4.2}
  \Delta_{1,2}(\lambda)=\cos\lambda\pi+\int_{-\pi}^{\pi}u_{1,2}(x)\exp(i\lambda x)dx\quad
\end{align}
for some functions $u_{1,1}(x)$ and $u_{1,2}(x)$ in $L_{2}(-\pi,\pi)$, respectively.
\end{lemma}

Next, we give the proof of Theorem \ref{th3.1}.

\textbf{Proof of Theorem \ref{th3.1}.} For the necessity, the asymptotic (\ref{1.3}) was already established in Theorem \ref{th1} for $\nu=1$. According to Lemma \ref{lem4.1}, the characteristic functions $\Delta_{1,1}(\lambda)$ and $\Delta_{1,2}(\lambda)$ have the representations (\ref{1.4-1}) and (\ref{1.4.1-1}). Thus, condition (ii) easily follows from representations (\ref{2.11.1}) and (\ref{2.12.1}).

For the sufficiency, we construct the functions  $\Delta_{1,1}(\lambda)$ and $\Delta_{1,2}(\lambda)$ by the representations (\ref{1.4-1}) and (\ref{1.4.1-1}) using the given sequences $\{\lambda_{n,1,1}\}_{n\in\mathbb{Z}}$ and $\{\lambda_{n,1,2}\}_{n\in\mathbb{Z}}$. According to Lemma \ref{lem5},  these functions have the forms (\ref{4.1}) and (\ref{4.2}) with some functions $u_{1,1}(x)$, $u_{1,2}(x)\in L_{2}(-\pi,\pi)$, respectively. Further, condition (ii) along with the Paley-Wiener theorem implies $u_{1,1}(x)=0$ and $u_{1,2}(x)=0$ a.e. on $(-\pi, a-\pi)\cup(\pi-a,\pi)$, that is, representations (\ref{2.11.1}) and  (\ref{2.12.1}) hold. First, construct the functions $w_{1,1}(x)$ and $w_{1,2}(x)$ by the formulae (\ref{3.1.1}) and (\ref{3.2.1}). Afterwards, construct the functions $q|_{[a,\frac{3a}{2}]\cup[\pi-\frac{a}{2},\pi]}$ and $p|_{[a,\frac{3a}{2}]\cup[\pi-\frac{a}{2},\pi]}$ by the formula (\ref{3.3.1}). Finally, construct the functions $q|_{(\frac{3a}{2},\pi-\frac{a}{2})}$ and $p|_{(\frac{3a}{2},\pi-\frac{a}{2})}$ by the formulae (\ref{3.6.1}) and (\ref{3.7.1}). Consider the corresponding problems $B_{1,1}(Q)$ and $B_{1,2}(Q)$. Then, as in Sections 2 and 3, one can show that $\Delta_{1,1}(\lambda)$ and $\Delta_{1,2}(\lambda)$ are their characteristic functions, respectively. $\square$

In view of the proof of Theorems \ref{th3.1}, we have the following algorithm for solving Inverse Problem 1.

\textbf{Algorithm 1.}  Let two spectra $\{\lambda_{n,1,1}\}_{n\in\mathbb{Z}}$ and $\{\lambda_{n,1,2}\}_{n\in\mathbb{Z}}$ be given.

(i) Construct the functions $\Delta_{1,1}(\lambda)$ and $\Delta_{1,2}(\lambda)$ by (\ref{1.4-1}) and (\ref{1.4.1-1});

(ii) In accordance with (\ref{2.11.1}) and (\ref{2.12.1}), find the functions $u_{1,1}(x)$ and $u_{1,2}(x)$  by the formulae
\begin{align}
u_{1,1}(x)=\frac{1}{2\pi}\sum_{n=-\infty}^{\infty}\Delta_{1,1}(n)\exp(-inx), \qquad\qquad\nonumber\\ u_{1,2}(x)=\frac{1}{2\pi}\sum_{n=-\infty}^{\infty}\Big(\Delta_{1,2}(n)-(-1)^{n}\Big)\exp(-inx);\!\!\!\!\!\!\nonumber
\end{align}

(iii) Construct the functions $w_{1,1}(x)$ and $w_{1,2}(x)$ by the formulae (\ref{3.1.1}) and (\ref{3.2.1});

(iv) Construct the functions $q|_{[a,\frac{3a}{2}]\cup[\pi-\frac{a}{2},\pi]}$ and $p|_{[a,\frac{3a}{2}]\cup[\pi-\frac{a}{2},\pi]}$ by the formula (\ref{3.3.1});

(v) Construct the functions $q|_{(\frac{3a}{2},\pi-\frac{a}{2})}$ and $p|_{(\frac{3a}{2},\pi-\frac{a}{2})}$ by the formulae (\ref{3.6.1}) and (\ref{3.7.1}).

\qquad

To prove Theorem \ref{th3}, we need the following lemma, which is similar to the proof of Lemma 3 in \cite{Buter2}.

\begin{lemma}\label{lem5.2}
For any complex sequences $\{\lambda_{n,2,1}\}_{n\in\mathbb{Z}}$ and $\{\lambda_{n,2,2}\}_{n\in\mathbb{Z}}$ of the form (\ref{1.3}) for $\nu=2$,
the functions $\Delta_{2,1}(\lambda)$ and $\Delta_{2,2}(\lambda)$ constructed by the formulae in (\ref{1.4}) and (\ref{1.4.1}) have the following forms:
\begin{align}\label{4.1.2}
  \Delta_{2,1}(\lambda)=\cos\lambda\pi+\int_{-\pi}^{\pi}u_{2,1}(x)\exp(i\lambda x)dx,\;\;
\end{align}
\begin{align}\label{4.2.2}
  \Delta_{2,2}(\lambda)=\sin\lambda\pi+\int_{-\pi}^{\pi}u_{2,2}(x)\exp(i\lambda x)dx\quad
\end{align}
for some functions $u_{2,1}(x)$ and $u_{2,2}(x)$ in $L_{2}(-\pi,\pi)$, respectively.
\end{lemma}

Next, we give the proof of Theorem \ref{th3}.

\textbf{Proof of Theorem \ref{th3}.} For the necessity, the asymptotic (\ref{1.3}) was already established in Theorem \ref{th1} for $\nu=2$. According to Lemma \ref{lem4}, the characteristic functions $\Delta_{2,1}(\lambda)$ and $\Delta_{2,2}(\lambda)$ have the representations (\ref{1.4}) and (\ref{1.4.1}). Thus, condition (ii) easily follows from representations (\ref{2.11}) and (\ref{2.12}).

For the sufficiency, we construct the functions  $\Delta_{2,1}(\lambda)$ and $\Delta_{2,2}(\lambda)$ by the representations (\ref{1.4}) and (\ref{1.4.1}) using the given sequences $\{\lambda_{n,2,1}\}_{n\in\mathbb{Z}}$ and $\{\lambda_{n,2,2}\}_{n\in\mathbb{Z}}$. According to Lemma \ref{lem5.2},  these functions have the forms (\ref{4.1.2}) and (\ref{4.2.2}) with some functions $u_{2,1}(x)$, $u_{2,2}(x)\in L_{2}(-\pi,\pi)$, respectively. Further, condition (ii) along with the Paley-Wiener theorem implies $u_{2,1}(x)=0$ and $u_{2,2}(x)=0$ a.e. on $(-\pi, a-\pi)\cup(\pi-a,\pi)$, that is, representations (\ref{2.11}) and  (\ref{2.12}) hold. First, construct the functions $w_{2,1}(x)$ and $w_{2,2}(x)$ by the formulae (\ref{3.1}) and (\ref{3.2}). Afterwards, construct the functions $q|_{[a,\frac{3a}{2}]\cup[\pi-\frac{a}{2},\pi]}$ and $p|_{[a,\frac{3a}{2}]\cup[\pi-\frac{a}{2},\pi]}$ by the formula (\ref{3.3}). Finally, construct the functions $q|_{(\frac{3a}{2},\pi-\frac{a}{2})}$ and $p|_{(\frac{3a}{2},\pi-\frac{a}{2})}$ by the formulae (\ref{3.6}) and (\ref{3.7}). Consider the corresponding problems $B_{2,1}(Q)$ and $B_{2,2}(Q)$. Then, as in Sections 2 and 3, one can show that $\Delta_{2,1}(\lambda)$ and $\Delta_{2,2}(\lambda)$ are their characteristic functions, respectively. $\square$

In view of the proof of Theorems \ref{th3}, we have the following algorithm for solving Inverse Problem 2.

\textbf{Algorithm 2.}  Let two spectra $\{\lambda_{n,2,1}\}_{n\in\mathbb{Z}}$ and $\{\lambda_{n,2,2}\}_{n\in\mathbb{Z}}$ be given.

(i) Construct the functions $\Delta_{2,1}(\lambda)$ and $\Delta_{2,2}(\lambda)$ by (\ref{1.4}) and (\ref{1.4.1});

(ii) In accordance with (\ref{2.11}) and (\ref{2.12}), find the functions $u_{2,1}(x)$ and $u_{2,2}(x)$  by the formulae
\begin{align}
 u_{2,1}(x)=\frac{1}{2\pi}\sum_{n=-\infty}^{\infty}\Big(\Delta_{2,1}(n)-(-1)^{n}\Big)\exp(-inx),\!\!\!\!\!\!\nonumber\\ u_{2,2}(x)=\frac{1}{2\pi}\sum_{n=-\infty}^{\infty}\Delta_{2,2}(n)\exp(-inx);\qquad\qquad\nonumber
\end{align}

(iii) Construct the functions $w_{2,1}(x)$ and $w_{2,2}(x)$ by the formulae (\ref{3.1}) and (\ref{3.2});

(iv) Construct the functions $q|_{[a,\frac{3a}{2}]\cup[\pi-\frac{a}{2},\pi]}$ and $p|_{[a,\frac{3a}{2}]\cup[\pi-\frac{a}{2},\pi]}$ by the formula (\ref{3.3});

(v) Construct the functions $q|_{(\frac{3a}{2},\pi-\frac{a}{2})}$ and $p|_{(\frac{3a}{2},\pi-\frac{a}{2})}$ by the formulae (\ref{3.6}) and (\ref{3.7}).

\section{stability of inverse problem}

Let us formulate the uniform stability for Inverse Problems 1 and 2, respectively. We use the symbol $C_{r}$ for denoting the constant depended only on $r$.

\begin{theorem}\label{th4}
For any fixed $r\in(0,\frac{1}{2})$, there exists $C_{r}>0$ such that the estimate
\begin{align}\label{5.1}
  \|q-\widetilde{q}\|_{L_{2}(a,\pi)}+\|p-\widetilde{p}\|_{L_{2}(a,\pi)}   \qquad\qquad\qquad\qquad\qquad\nonumber\\
  \leq C_{r}\Big(\|\{\lambda_{n,1,1}\!-\!\widetilde{\lambda}_{n,1,1}\}_{n\in\mathbb{Z}}\|_{l_{2}}\!+\!
  \|\{\lambda_{n,1,2}\!-\!\widetilde{\lambda}_{n,1,2}\}_{n\in\mathbb{Z}}\|_{l_{2}} \Big)\!\!\!
\end{align}
is fulfilled as soon as $\|\{\lambda_{n,1,j}\!-n+\frac{j-1}{2}\}_{n\in\mathbb{Z}}\|_{l_{2}}\!\leq\! r$ and $\|\{\widetilde{\lambda}_{n,1,j}\!-n+\frac{j-1}{2}\}_{n\in\mathbb{Z}}\|_{l_{2}}\!\leq\! r$ for $j=1,2$.
\end{theorem}

\begin{theorem}\label{th4.2}
For any fixed $r\in(0,\frac{1}{2})$, there exists $C_{r}>0$ such that the estimate
\begin{align}\label{5.2}
  \|q-\widetilde{q}\|_{L_{2}(a,\pi)}+\|p-\widetilde{p}\|_{L_{2}(a,\pi)}   \qquad\qquad\qquad\qquad\qquad\nonumber\\
  \leq C_{r}\Big(\|\{\lambda_{n,2,1}\!-\!\widetilde{\lambda}_{n,2,1}\}_{n\in\mathbb{Z}}\|_{l_{2}}\!+\!
  \|\{\lambda_{n,2,2}\!-\!\widetilde{\lambda}_{n,2,2}\}_{n\in\mathbb{Z}}\|_{l_{2}} \Big)\!\!\!
\end{align}
is fulfilled as soon as $\|\{\lambda_{n,2,j}\!-n+\frac{j}{2}\}_{n\in\mathbb{Z}}\|_{l_{2}}\!\leq\! r$ and $\|\{\widetilde{\lambda}_{n,2,j}\!-n+\frac{j}{2}\}_{n\in\mathbb{Z}}\|_{l_{2}}\!\leq\! r$ for $j=1,2$.
\end{theorem}

To prove Theorems \ref{th4} and \ref{th4.2}, we need the following lemma, which immediately follows from Theorem 7 in \cite{Buter1}.

\begin{lemma}\label{lem6}
Let $\nu,j\in\{1,2\}$. For any fixed $r>0$, there exists $C_{r}>0$ such that the estimate
\begin{align}
  \|u_{\nu,j}-\widetilde{u}_{\nu,j}\|_{L_{2}(a-\pi,\pi-a)}\leq C_{r}\|\{\lambda_{n,\nu,,j}-\widetilde{\lambda}_{n,\nu,j}\}_{n\in\mathbb{Z}}\|_{l_{2}} \nonumber
\end{align}
is fulfilled as soon as $\|\{\lambda_{n,\nu,j}\!-n-\frac{2-\nu-j}{2}\}_{n\in\mathbb{Z}}\|_{l_{2}}\!\leq\! r$ and $\|\{\widetilde{\lambda}_{n,\nu,j}\!-n-\frac{2-\nu-j}{2}\}_{n\in\mathbb{Z}}\|_{l_{2}}\!\leq\! r$.
\end{lemma}

Now we are in position to give the proof of Theorem \ref{th4}.

\textbf{Proof of Theorem \ref{th4}.} According Lemma \ref{lem3.1}, we have
\begin{equation}
  |\gamma_{1,1}(x)|\leq 2\|w_{1,1}\|_{L_{2}(a,\pi)}\|w_{1,2}\|_{L_{2}(a,\pi)}, \;\;x\in(\frac{3a}{2},\pi-\frac{a}{2}),\nonumber
\end{equation}
then
\begin{align}
  \int_{\frac{3a}{2}}^{\pi-\frac{a}{2}}|q(x)|^{2}dx=\int_{\frac{3a}{2}}^{\pi-\frac{a}{2}}|w_{1,1}(x)+\gamma_{1,1}(x)|^{2}dx \quad\:\nonumber\\
  \leq 2\int_{\frac{3a}{2}}^{\pi-\frac{a}{2}}|w_{1,1}(x)|^{2}dx+2\int_{\frac{3a}{2}}^{\pi-\frac{a}{2}}|\gamma_{1,1}(x)|^{2}dx \qquad\nonumber\\
  \leq 2\int_{\frac{3a}{2}}^{\pi-\frac{a}{2}}|w_{1,1}(x)|^{2}dx+8\pi\|w_{1,1}\|_{L_{2}(a,\pi)}^{2}\|w_{1,2}\|_{L_{2}(a,\pi)}^{2}.\!\!\!\!\!\!\!\!\!\!
  \nonumber
\end{align}
This along with Lemma \ref{lem2.1} yield
\begin{align}
  \int_{a}^{\pi}|q(x)|^{2}dx=\int_{a}^{\frac{3a}{2}}|q(x)|^{2}dx+\int_{\frac{3a}{2}}^{\pi-\frac{a}{2}}|q(x)|^{2}dx
  +\int_{\pi-\frac{a}{2}}^{\pi}|q(x)|^{2}dx \!\nonumber\\
  \leq\int_{a}^{\pi}|w_{1,1}(x)|^{2}dx+8\pi\|w_{1,1}\|_{L_{2}(a,\pi)}^{2}\|w_{1,2}\|_{L_{2}(a,\pi)}^{2}\quad\;\;\nonumber\\
  \leq8\pi\|w_{1,1}\|_{L_{2}(a,\pi)}^{2}(1+\|w_{1,2}\|_{L_{2}(a,\pi)}^{2}),\qquad\qquad\qquad\quad\!\nonumber
\end{align}
which implies
\begin{align}\label{5.1}
  \|q\|_{L_{2}(a,\pi)}\leq 2\sqrt{2\pi}\|w_{1,1}\|_{L_{2}(a,\pi)}\sqrt{1+\|w_{1,2}\|_{L_{2}(a,\pi)}^{2}}\qquad\qquad\qquad\!\!\nonumber\\
  \leq 2\sqrt{2\pi}\|w_{1,1}\|_{L_{2}(a,\pi)}(1+\|w_{1,2}\|_{L_{2}(a,\pi)})\qquad\qquad\qquad\!\nonumber\\
  \leq 2\sqrt{2\pi}(\|w_{1,1}\|_{L_{2}(a,\pi)}+\|w_{1,1}\|_{L_{2}(a,\pi)}^{2}+\|w_{1,2}\|_{L_{2}(a,\pi)}^{2}).\!\!
\end{align}
Similarly, we get
\begin{align}\label{5.2}
   \|p\|_{L_{2}(a,\pi)}\!\leq 2\sqrt{2\pi}(\|w_{1,2}\|_{L_{2}(a,\pi)}\!+\!\|w_{1,1}\|_{L_{2}(a,\pi)}^{2}\!+\!\|w_{1,2}\|_{L_{2}(a,\pi)}^{2}).
\end{align}

Using the estimates (\ref{3.2-1}), (\ref{5.1}) and (\ref{5.2}), we obtain

\begin{align}\label{5.3}
   \|q\|_{L_{2}(a,\pi)}\!+\!\|p\|_{L_{2}(a,\pi)}\!\leq \!16\pi\Big(\|u_{1,1}\|_{L_{2}(a-\pi,\pi-a)}\!+\!\|u_{1,2}\|_{L_{2}(a-\pi,\pi-a)} \!\!\!\! \nonumber\\
   \!+\|u_{1,1}\|_{L_{2}(a-\pi,\pi-a)}^{2}\!+\!\|u_{1,2}\|_{L_{2}(a-\pi,\pi-a)}^{2}\Big).
\end{align}
According Lemma \ref{lem6}, for any fixed $r\in(0,\frac{1}{2})$, there exists $C_{r}>0$ such that the estimate
\begin{align}\label{5.4}
  \|u_{1,j}-\widetilde{u}_{1,j}\|_{L_{2}(a-\pi,\pi-a)}\leq C_{r}\|\{\lambda_{n,1,j}-\widetilde{\lambda}_{n,1,j}\}_{n\in\mathbb{Z}}\|_{l_{2}}
\end{align}
is fulfilled as soon as $\|\{\lambda_{n,1,j}\!-n+\frac{j-1}{2}\}_{n\in\mathbb{Z}}\|_{l_{2}}\!\leq\! r$ and $\|\{\widetilde{\lambda}_{n,1,j}\!-n+\frac{j-1}{2}\}_{n\in\mathbb{Z}}\|_{l_{2}}\!\leq\! r$.

Next, we use one and the same symbol $C_{r}$ for denoting different positive constant depended only on $r$. Using the estimates (\ref{5.3}) and
(\ref{5.4}), we obtain
\begin{align}\label{5.5}
 \|q-\widetilde{q}\|_{L_{2}(a,\pi)}\!+\!\|p-\widetilde{p}\|_{L_{2}(a,\pi)}\qquad\qquad\qquad\qquad\quad\qquad\! \nonumber\\
 \leq \!C_{r}\Big(\|u_{1,1}-\widetilde{u}_{1,1}\|_{L_{2}(a-\pi,\pi-a)}\!+\!\|u_{1,2}-\widetilde{u}_{1,2}\|_{L_{2}(a-\pi,\pi-a)}\nonumber\\
 \!+\|u_{1,1}-\widetilde{u}_{1,1}\|_{L_{2}(a-\pi,\pi-a)}^{2}\!+\!\|u_{1,2}-\widetilde{u}_{1,2}\|_{L_{2}(a-\pi,\pi-a)}^{2}\Big)\!\!\!\!\!\nonumber\\
 \leq \!C_{r}\Big(\|\{\lambda_{n,1,1}-\widetilde{\lambda}_{n,1,1}\}_{n\in\mathbb{Z}}\|_{l_{2}}\!
 +\!\|\{\lambda_{n,1,2}-\widetilde{\lambda}_{n,1,2}\}_{n\in\mathbb{Z}}\|_{l_{2}} \, \nonumber\\
 \!+\|\{\lambda_{n,1,1}-\widetilde{\lambda}_{n,1,1}\}_{n\in\mathbb{Z}}\|_{l_{2}}^{2}\!
 +\!\|\{\lambda_{n,1,2}-\widetilde{\lambda}_{n,1,2}\}_{n\in\mathbb{Z}}\|_{l_{2}}^{2}\Big)\!\!\!
\end{align}
is fulfilled as soon as $\|\{\lambda_{n,1,j}\!-n+\frac{j-1}{2}\}_{n\in\mathbb{Z}}\|_{l_{2}}\!\leq\! r$ and $\|\{\widetilde{\lambda}_{n,1,j}\!-n+\frac{j-1}{2}\}_{n\in\mathbb{Z}}\|_{l_{2}}\!\leq\! r$.

Since $\|\{\lambda_{n,1,j}\!-n+\frac{j-1}{2}\}_{n\in\mathbb{Z}}\|_{l_{2}}\!\leq\! r<\frac{1}{2}$ and $\|\{\widetilde{\lambda}_{n,1,j}\!-n+\frac{j-1}{2}\}_{n\in\mathbb{Z}}\|_{l_{2}}\!\leq\! r<\frac{1}{2}$, then
\begin{align}
 \|\{\lambda_{n,1,j}\!-\!\widetilde{\lambda}_{n,1,j}\}_{n\in\mathbb{Z}}\|_{l_{2}}\!\leq\! 2r<1,\;\; j=1,2, \nonumber
\end{align}
which implies
\begin{align}\label{5.6}
 \|\{\lambda_{n,1,j}\!-\!\widetilde{\lambda}_{n,1,j}\}_{n\in\mathbb{Z}}\|_{l_{2}}^{2}\!
 \leq\!\|\{\lambda_{n,1,j}\!-\!\widetilde{\lambda}_{n,1,j}\}_{n\in\mathbb{Z}}\|_{l_{2}}, \;\; j=1,2.
\end{align}
In view of the estimates (\ref{5.5}) and (\ref{5.6}), we arrive at the assertion of Theorem  \ref{th4}. $\square$

Finally, we give the proof of Theorem \ref{th4.2}.

\textbf{Proof of Theorem \ref{th4.2}.} According Lemma \ref{lem3}, we have
\begin{equation}
  |\gamma_{2,1}(x)|\leq 2\|w_{2,1}\|_{L_{2}(a,\pi)}\|w_{2,2}\|_{L_{2}(a,\pi)}, \;\;x\in(\frac{3a}{2},\pi-\frac{a}{2}),\nonumber
\end{equation}
then
\begin{align}
  \int_{\frac{3a}{2}}^{\pi-\frac{a}{2}}|q(x)|^{2}dx=\int_{\frac{3a}{2}}^{\pi-\frac{a}{2}}|w_{2,1}(x)-\gamma_{2,1}(x)|^{2}dx \quad\:\nonumber\\
  \leq 2\int_{\frac{3a}{2}}^{\pi-\frac{a}{2}}|w_{2,1}(x)|^{2}dx+2\int_{\frac{3a}{2}}^{\pi-\frac{a}{2}}|\gamma_{2,1}(x)|^{2}dx \qquad\nonumber\\
  \leq 2\int_{\frac{3a}{2}}^{\pi-\frac{a}{2}}|w_{2,1}(x)|^{2}dx+8\pi\|w_{2,1}\|_{L_{2}(a,\pi)}^{2}\|w_{2,2}\|_{L_{2}(a,\pi)}^{2}.\!\!\!\!\!\!\!\!
  \nonumber
\end{align}
This along with Lemma \ref{lem2} yield
\begin{align}
  \int_{a}^{\pi}|q(x)|^{2}dx=\int_{a}^{\frac{3a}{2}}|q(x)|^{2}dx+\int_{\frac{3a}{2}}^{\pi-\frac{a}{2}}|q(x)|^{2}dx
  +\int_{\pi-\frac{a}{2}}^{\pi}|q(x)|^{2}dx \!\!\nonumber\\
  \leq\int_{a}^{\pi}|w_{2,1}(x)|^{2}dx+8\pi\|w_{2,1}\|_{L_{2}(a,\pi)}^{2}\|w_{2,2}\|_{L_{2}(a,\pi)}^{2}\quad\;\:\nonumber\\
  \leq8\pi\|w_{2,1}\|_{L_{2}(a,\pi)}^{2}(1+\|w_{2,2}\|_{L_{2}(a,\pi)}^{2}),\qquad\qquad\qquad\quad\!\nonumber
\end{align}
which implies
\begin{align}\label{5.1.2}
  \|q\|_{L_{2}(a,\pi)}\leq 2\sqrt{2\pi}\|w_{2,1}\|_{L_{2}(a,\pi)}\sqrt{1+\|w_{2,2}\|_{L_{2}(a,\pi)}^{2}}\qquad\qquad\qquad\!\!\nonumber\\
  \leq 2\sqrt{2\pi}\|w_{2,1}\|_{L_{2}(a,\pi)}(1+\|w_{2,2}\|_{L_{2}(a,\pi)})\qquad\qquad\qquad\!\nonumber\\
  \leq 2\sqrt{2\pi}(\|w_{2,1}\|_{L_{2}(a,\pi)}+\|w_{2,1}\|_{L_{2}(a,\pi)}^{2}+\|w_{2,2}\|_{L_{2}(a,\pi)}^{2}).\!\!
\end{align}
Similarly, we get
\begin{align}\label{5.2.2}
   \|p\|_{L_{2}(a,\pi)}\!\leq 2\sqrt{2\pi}(\|w_{2,2}\|_{L_{2}(a,\pi)}\!+\!\|w_{2,1}\|_{L_{2}(a,\pi)}^{2}\!+\!\|w_{2,2}\|_{L_{2}(a,\pi)}^{2}).
\end{align}

Using the estimates (\ref{3.2-2}), (\ref{5.1.2}) and (\ref{5.2.2}), we obtain

\begin{align}\label{5.3.2}
   \|q\|_{L_{2}(a,\pi)}\!+\!\|p\|_{L_{2}(a,\pi)}\!\leq \!16\pi\Big(\|u_{2,1}\|_{L_{2}(a-\pi,\pi-a)}\!+\!\|u_{2,2}\|_{L_{2}(a-\pi,\pi-a)} \!\!\!\! \nonumber\\
   \!+\|u_{2,1}\|_{L_{2}(a-\pi,\pi-a)}^{2}\!+\!\|u_{2,2}\|_{L_{2}(a-\pi,\pi-a)}^{2}\Big).
\end{align}
According Lemma \ref{lem6}, for any fixed $r\in(0,\frac{1}{2})$, there exists $C_{r}>0$ such that the estimate
\begin{align}\label{5.4.2}
  \|u_{2,j}-\widetilde{u}_{2,j}\|_{L_{2}(a-\pi,\pi-a)}\leq C_{r}\|\{\lambda_{n,2,j}-\widetilde{\lambda}_{n,2,j}\}_{n\in\mathbb{Z}}\|_{l_{2}}
\end{align}
is fulfilled as soon as $\|\{\lambda_{n,2,j}\!-n+\frac{j}{2}\}_{n\in\mathbb{Z}}\|_{l_{2}}\!\leq\! r$ and $\|\{\widetilde{\lambda}_{n,2,j}\!-n+\frac{j}{2}\}_{n\in\mathbb{Z}}\|_{l_{2}}\!\leq\! r$.

Next, we use one and the same symbol $C_{r}$ for denoting different positive constant depended only on $r$. Using the estimates (\ref{5.3.2}) and
(\ref{5.4.2}), we obtain
\begin{align}\label{5.5.2}
 \|q-\widetilde{q}\|_{L_{2}(a,\pi)}\!+\!\|p-\widetilde{p}\|_{L_{2}(a,\pi)}\qquad\qquad\qquad\qquad\quad\qquad\!\nonumber\\
 \leq \!C_{r}\Big(\|u_{2,1}-\widetilde{u}_{2,1}\|_{L_{2}(a-\pi,\pi-a)}\!+\!\|u_{2,2}-\widetilde{u}_{2,2}\|_{L_{2}(a-\pi,\pi-a)}\nonumber\\
 \!+\|u_{2,1}-\widetilde{u}_{2,1}\|_{L_{2}(a-\pi,\pi-a)}^{2}\!+\!\|u_{2,2}-\widetilde{u}_{2,2}\|_{L_{2}(a-\pi,\pi-a)}^{2}\Big)\!\!\!\!\!\nonumber\\
 \leq \!C_{r}\Big(\|\{\lambda_{n,2,1}-\widetilde{\lambda}_{n,2,1}\}_{n\in\mathbb{Z}}\|_{l_{2}}\!
 +\!\|\{\lambda_{n,2,2}-\widetilde{\lambda}_{n,2,2}\}_{n\in\mathbb{Z}}\|_{l_{2}} \, \nonumber\\
 \!+\|\{\lambda_{n,2,1}-\widetilde{\lambda}_{n,2,1}\}_{n\in\mathbb{Z}}\|_{l_{2}}^{2}\!
 +\!\|\{\lambda_{n,2,2}-\widetilde{\lambda}_{n,2,2}\}_{n\in\mathbb{Z}}\|_{l_{2}}^{2}\Big)\!\!\!
\end{align}
is fulfilled as soon as $\|\{\lambda_{n,2,j}\!-n+\frac{j}{2}\}_{n\in\mathbb{Z}}\|_{l_{2}}\!\leq\! r$ and $\|\{\widetilde{\lambda}_{n,2,j}\!-n+\frac{j}{2}\}_{n\in\mathbb{Z}}\|_{l_{2}}\!\leq\! r$.

Since $\|\{\lambda_{n,2,j}\!-n+\frac{j}{2}\}_{n\in\mathbb{Z}}\|_{l_{2}}\!\leq\! r<\frac{1}{2}$ and $\|\{\widetilde{\lambda}_{n,2,j}\!-n+\frac{j}{2}\}_{n\in\mathbb{Z}}\|_{l_{2}}\!\leq\! r<\frac{1}{2}$, then
\begin{align}
 \|\{\lambda_{n,2,j}\!-\!\widetilde{\lambda}_{n,2,j}\}_{n\in\mathbb{Z}}\|_{l_{2}}\!\leq\! 2r<1,\;\; j=1,2, \nonumber
\end{align}
which implies
\begin{align}\label{5.6.2}
 \|\{\lambda_{n,2,j}\!-\!\widetilde{\lambda}_{n,2,j}\}_{n\in\mathbb{Z}}\|_{l_{2}}^{2}\!
 \leq\!\|\{\lambda_{n,2,j}\!-\!\widetilde{\lambda}_{n,2,j}\}_{n\in\mathbb{Z}}\|_{l_{2}}, \;\; j=1,2.
\end{align}
In view of the estimates (\ref{5.5.2}) and (\ref{5.6.2}), we arrive at the assertion of Theorem \ref{th4.2}. $\square$

\qquad
	
	%%%%%%%%%%%%%%%%%%%%%%%%%%%%%%%%%%%%%
	\noindent {\bf Acknowledgments.}
	This work was supported in part by  the National Natural Science Foundation of China (11871031) and the National Natural Science Foundation of Jiang Su (BK20201303).

\end{document}